\documentclass[a4paper]{amsart}

\usepackage{amsmath,amssymb,amsthm,amsfonts}
\usepackage[UKenglish]{babel}
\usepackage{microtype}
\usepackage{enumitem}
\usepackage[colorlinks=false, pdfborder={0 0 0}]{hyperref}

\newtheorem{theorem}{Theorem}[section]

\newtheorem{proposition}[theorem]{Proposition}
\newtheorem{corollary}[theorem]{Corollary}
\theoremstyle{definition}
\newtheorem{definition}[theorem]{Definition}
\newtheorem{example}[theorem]{Example}

\theoremstyle{remark}
\newtheorem{remark}[theorem]{Remark}

\numberwithin{equation}{section}

\def\Sz{\operatorname{Sz}}
\def\iten{\widehat{\otimes}_\varepsilon}
\def\vspan{\operatorname{span}}
\def\cvspan{\closure{\vspan}}
\def\ran{\operatorname{ran}}
\def\E{\mathsf{ E}}
\def\F{\mathsf{ F}}
\def\G{\mathsf{ G}}
\def\Ed{{\E^*}}

\def\N{\mathbb{ N}}
\def\LL{\mathcal{ L}}
\def\KK{\mathcal{ K}}
\def\NN{\mathcal{ N}}

\def\<{\langle}
\def\>{\rangle}
\providecommand{\norm}[1]{\left\lVert#1\right\rVert}

\def\eqnorm{|||\hspace*{.35em} |||}

\newcommand{\closure}[2][3]{%
  {}\mkern#1mu\overline{\mkern-#1mu#2}}

\def\d{\overline{\delta}}
\def\p{\overline{\rho}}

\newcommand{\sd}[1]{{\hat{\delta}_{#1}}}
\renewcommand{\sp}[1]{{\hat{\rho}_{#1}}}

\begin{document}

\title{On strong asymptotic uniform smoothness and convexity}

\author{Luis Garc\'ia-Lirola}
\address[L.~Garc\'ia-Lirola and M.~Raja] {Departamento de Matem\'aticas, Facultad de Matem\'aticas, Universidad de Murcia,\\ 
30100 Espinardo (Murcia), Spain}
\email[L.~Garc\'ia-Lirola]{luiscarlos.garcia@um.es}
\author{Mat\'ias Raja}
\email[M.~Raja]{matias@um.es}

\thanks{Partially supported by the grants MINECO/FEDER MTM2014-57838-C2-1-P and Fundaci\'on S\'eneca CARM
19368/PI/14}
\date{February, 2017}



\begin{abstract} We introduce the notions of strong asymptotic uniform smoothness and convexity. We show that the injective tensor product of strongly asymptotically uniformly smooth spaces is asymptotically uniformly smooth. This applies in particular to uniformly smooth spaces admitting a monotone FDD, extending a result by Dilworth, Kutzarova, Randrianarivony, Revalski and Zhivkov~\cite{DKLRZ13}. Our techniques also provide a characterisation of Orlicz functions $M, N$ such that the space of compact operators $\KK(h_M,h_N)$ is asymptotically uniformly smooth. Finally we show that $\KK(X, Y)$ is not strictly convex whenever $X$ and $Y$ are at least two-dimensional, which extends a result by Dilworth and Kutzarova~\cite{DK95}.
\end{abstract}

\maketitle

\section{Introduction and notation}\label{sec:intro}

Consider a real Banach space $X$ and let $S_X$ be its unit sphere. For $t>0$, $x\in S_X$ we shall consider
\begin{align*}
\d_X(t,x) &= \sup_{\dim(X/Y)<\infty}\inf_{y\in S_Y} \norm{x+ty}-1;\\
\p_X(t,x) &= \inf_{\dim(X/Y)<\infty}\sup_{y\in S_Y} \norm{x+ty}-1\,.
\end{align*}
The \emph{modulus of asymptotic uniform convexity} of $X$ is given by
\[ \d_X(t) = \inf_{x\in S_X} \d_X(t,x)\,, \]
and \emph{modulus of asymptotic uniform smoothness} of $X$ is given by
\[ \p_X(t) = \sup_{x\in S_X} \p_X(t,x)\,. \]
The space $X$ is said to be \emph{asymptotically uniformly convex} (AUC for short) if $\d_X(t)>0$ for each $t>0$ and it is said to be \emph{asymptotically uniformly smooth} (AUS for short) if $\lim_{t\to 0} t^{-1}\p_X(t) = 0$.
 If $X$ is a dual space and we considered only weak* closed subspaces of $X$ then the corresponding modulus is denoted by $\d_X^*(t)$. 
The space $X$ is said to be \emph{weak* asymptotically uniformly convex} if $\d_X^*(t)>0$ for each $t>0$. Let us highlight that it is proved in~\cite{DKLR16} that a space is AUS if and only if its dual space is weak* AUC. In addition,  $\p_X$ is quantitatively related to $\d_{X}^*$ by Young's duality. We refer the reader to~\cite{JLPS02} and the references therein for a detailed study of these properties. The related notions of \emph{nearly uniformly convex} space (NUC for short) and \emph{nearly uniformly smooth} (NUS for short) were introduced by Huff~\cite{H80} and Prus~\cite{P89}. A space is NUS if and only if it is AUS and reflexive and if and only if its dual is NUC. 

If $X$, $Y$ are Banach spaces, we denote $\KK(X,Y)$ the space of compact operators from $X$ to $Y$ endowed with the operator norm. In addition, we denote $\NN(X,Y)$ the space of nuclear operators endowed with the nuclear norm. The tensor product $X \otimes Y$ can be identified with the space finite rank operators from $Y^*$ to $X$. That is, given $u = \sum_{i=1}^n x_i\otimes y_i$, $u(y^*) = ~\sum_{i=1}^n y^*(y_i)x_i$. By using the transposition mapping given by $x\otimes y\mapsto y\otimes x$, we can also consider $u$ as an operator from $X^*$ into $Y$. The completion of $X\otimes Y$ endowed with the operator norm is called the \emph{injective tensor product} and denoted by $X \iten Y$. Under suitable hypothesis (e.g.\ if there is an FDD for $Y$) the space $X^*\iten Y$ is isometric to $\KK(X, Y)$. Moreover, Grothendieck proved that if $Y^*$ has the RNP then $(X\iten Y)^*$ is isometric to $\NN(X, Y^*)$. For a comprehensive treatment on tensor products one may refer to the book~\cite{R02}. 

We will consider the following partial order for functions defined on $(0, 1]$. We write $f\preceq g$ if there is a constant $c > 0$ such that $f(t) \leq c g(t)$ for
all $t\in(0, 1]$. If $f\preceq g$ and $g\preceq f$, then we say that $f$ and $g$ are \emph{equivalent}. Given $1\leq p < \infty$, we will say that a modulus $\delta$ of convexity is of power type $p$ if $\delta\succeq t^p$, and that a modulus $\rho$ of smoothness is of power type $p$ if $\rho\preceq t^p$.

Lennard proved in~\cite{L90} that the space of trace class operators on a Hilbert space is weak* AUC. Equivalently, $\KK(\ell_2)$ is AUS. This result was extended by Besbes in~\cite{B92}, who showed that $\KK(\ell_p,\ell_p)$ is AUS whenever $1<p<\infty$. Moreover, in~\cite{DKLRZ13} it is proved that $\KK(\ell_p,\ell_q)$ is AUS with power type $\min\{p',q\}$ for every $1<p,q<\infty$. On the order hand, Causey recently showed in~\cite{C17} that the Szlenk index of $X\iten Y$ is equal to the maximum of the Szlenk indices of $X$ and $Y$ for all Banach spaces $X$ and $Y$. In particular, $X\iten Y$ admits an equivalent AUS norm if and only if $X$ and $Y$ do. Moreover, Draga and Kochanek have proved in~\cite{DK16} that is possible to get an equivalent AUS norm in $X\iten Y$ with power type the maximum of the ones of the norm of $X$ and $Y$. Nevertheless, it seems to be an open question if the injective tensor product of AUS spaces is an AUS space in its canonical norm. 

In this paper we introduce the notion of strongly AUC and strongly AUS spaces, and we show that the injective tensor product of strongly AUS spaces is AUS. In particular, our result applies to uniformly smooth spaces admitting monotone FDDs. 

\begin{theorem}\label{th:usitenaus}
Let $X$ and $Y$ be Banach spaces with monotone FDDs. If $X$ and $Y$ are uniformly smooth then $X\iten Y$ is AUS. Moreover,  
\[ \p_{X\iten Y} (t) \leq (1+ \frac{1}{2}\rho_X(4t))(1+\frac{1}{2}\rho_Y(4t))-1 \]
whenever $t\leq 1/4$. In particular, if $X$ is uniformly smooth with power type $p$ and $Y$ is uniformly smooth with power type $q$ then $X\iten Y$ is AUS with power type $\min\{p,q\}$.  
\end{theorem}

The previous result and the isometry between $\KK(X,Y)$ and $X^*\iten Y$ yields the generalisation of Theorem 4.3 in~\cite{DKLRZ13}.

\begin{theorem}\label{th:uskaus}
Let $X, Y$ be Banach spaces and assume that $X^*$ and $Y$ have monotone FDDs. If $X$ is uniformly convex and $Y$ is uniformly smooth then $\KK(X,Y)$ is AUS. Moreover, if $X$ is uniformly convex with power type $p$ and $Y$ is uniformly smooth with power type $q$ then $\KK(X,Y)$ is AUS with power type $\min\{p',q\}$.
\end{theorem}

Our techniques also lead to a characterisation of Orlicz functions $M$, $N$ such that the space $\KK(h_M, h_N)$ is AUS in terms of their Boyd indices $\alpha_M,\beta_M$ (see Section~\ref{sec:orlicz} for definitions). Namely, 

\begin{theorem}\label{th:kAUSorlicz} Let $M,N$ be Orlicz functions. The space $\KK(h_M,h_N)$ is AUS if and only if $\alpha_M,\alpha_N>1$ and $\beta_M<+\infty$. Moreover, $\min\{\beta_M',\alpha_N\}$ is the supremum of the numbers $\alpha>0$ such that the modulus of asymptotic smoothness of $\KK(h_M,h_N)$ is of power type $\alpha$.
\end{theorem}

Remark that, for the natural norm, not much can be expected. Indeed, Ruess and Stegall showed in~\cite[Corollary 3.5]{RS86} that neither the norm of $X\iten Y$ or the norm of $\KK(X,Y)$ are smooth whenever the dimension of $X$ and $Y$ are greater or equal than $2$. On the other hand, Dilworth and Kutzarova proved in~\cite{DK95} that $\LL(\ell_p,\ell_q)$ is not strictly convex for $1\leq p \leq q \leq \infty$. We have obtained the following result. 

\begin{proposition}\label{th:KKstrictlyconvex} Let $X$, $Y$ be Banach spaces with dimension greater or equal than $2$. Then $\KK(X,Y)$ and $X\iten Y$ are not strictly convex.
\end{proposition} 

Recall that a sequence $\E=(E_n)_n$ of finite dimensional subspaces of $X$ is call a \emph{finite dimensional decomposition} (FDD for short) if every $x\in X$ has a unique representation of the form $x=\sum_{n=1}^\infty x_n$, with $x_n\in E_n$ for every $n$. Every FDD of $X$ determines a sequence of uniformly bounded projections $(P_n^\E)_n$ given by $P_n^\E(\sum_{i=1}^\infty x_i) = \sum_{i=1}^n x_i$. The number $K=\sup\{\norm{P_n^\E}\}$ is called the decomposition constant of the FDD. An FDD is called \emph{monotone} if $K=1$. Moreover, an FDD is called \emph{shrinking} if $\lim_n \norm{P_n^*x^*-x^*}=0$ for every $x^*\in X^*$, and it is called \emph{boundedly complete} if $\sum_{n=1}^\infty x_n$ converges whenever $x_n\in E_n$ for each $n$ and $\sup_n \norm{\sum_{i\leq n} x_i}<+\infty$. We say that $\F=(F_n)_n$ is a \emph{blocking} of $\E$ if there exists an increasing sequence $(m_n)_n\subset \mathbb N$ such that $m_1=0$ and $F_n = \bigoplus_{i=m_{n}+1}^{m_{n+1}} E_i$ for every $n$. For detailed treatment and applications of FDDs, we refer the reader to~\cite{LT77}.

This paper is organized as follows. In the second section we introduce a notion of asymptotic moduli with respect to a norming subspace, which includes the usual asymptotic moduli, and we give a formula for these moduli in spaces having an FDD. This formula motivates the definition of strongly AUS and strongly AUC spaces, which is given in the third section together with their basic properties.  In the fourth section we show that the injective tensor product of strongly AUS spaces is strongly AUS, which allow us to prove Theorem~\ref{th:usitenaus} and~\ref{th:uskaus}. The fifth section is devoted to the study strong asymptotic uniform smoothness and convexity in the particular case of Orlicz and Lorentz sequence spaces, including the proof of Theorem~\ref{th:kAUSorlicz}. Finally, Section 6 includes the proof of Proposition~\ref{th:KKstrictlyconvex}.

\section{On \texorpdfstring{$F$}{F}-AUC and \texorpdfstring{$F$}{F}-AUS spaces}\label{sec:fauc}

Given $F$ a norming subspace of $X^*$, let us denote by $\sigma(X,F)$ the coarsest topology on $X$ with respect to which every element of $F$ is continuous. We shall introduce a general concept of $F$-AUC and $F$-AUS norms.

\begin{definition} Let $F$ be a norming subspace of $X^*$. For $t>0$ and $x\in S_X$, we define
\begin{align*}
\d_X^F(t,x) &= \sup_{\substack{\dim(X/Y)<\infty\\ Y \sigma(X,F)\text{-closed}}}\inf_{y\in S_Y} \norm{x+ty}-1;\\
\p_X^F(t,x) &= \inf_{\substack{\dim(X/Y)<\infty\\ Y \sigma(X,F)\text{-closed}}}\sup_{y\in S_Y} \norm{x+ty}-1\,.
\end{align*}
The corresponding moduli are defined as follows
\begin{align*}
\d_X^F(t) &= \inf_{x\in S_X}\d_X^F(t,x), & \p_X^F(t) &=\sup_{x\in S_X}\p_X^F(t,x) 
\end{align*}
The space $X$ is said to be \emph{$F$-asymptotically uniformly convex} if $\d_X^F(t)>0$ for each $t>0$ and it is said to be \emph{$F$-asymptotically uniformly smooth} if $\lim_{t\to 0} t^{-1}\p_X^F(t) = 0$.
\end{definition}

Note that $\d_X= \d_X^{X^*}$, $\p_X= \p_X^{X^*}$ and $\d_X^* = \d_{X^*}^X$. Thus, a space $X$ is AUC (resp. AUS) if and only if it is $X^*$-AUC (resp. $X^*$-AUS), and $X^*$ is weak* AUC if and only if it is $X$-AUC. 

Dutrieux showed in~\cite[Lemma 37]{D02} that if $X^*$ is separable then the modulus of asymptotic smoothness admits the following sequential expression: 
\begin{equation*} \p_X(t,x) = \sup_{\substack{x_n\stackrel{w}{\to}0\\ \norm{x_n}\leq t}} \limsup_{n\to\infty} \norm{x+x_n}-1 \,.
\end{equation*}
In addition, Borel-Mathurin proved in~\cite{BM10} that a similar statement for the weak* modulus of asymptotic convexity of $X^*$ holds when $X$ is separable. Namely,
\begin{equation*} \d^{*}_{X}(t,x^*) = \inf_{\substack{x^*_n\stackrel{w^*}{\to}0\\ \norm{x^*_n}\geq t}} \liminf_{n\to\infty} \norm{x^*+x^*_n}-1 \,.
\end{equation*}

The same ideas can be used to prove the following result, which can be seen as a general version of both formulas. Note that a finite codimensional subspace $Y$ of $X$ is $\sigma(X,F)$-closed if and only if there are $f_1,\ldots, f_n\in F$ such that $Y = \bigcap_{i=1}^n \ker f_i$. Moreover, if $(x_\alpha)_\alpha$ is a $\sigma(X,F)$-null net in $X$ then $\lim_\alpha d(x_\alpha, Y) = 0$ for each finite codimensional $\sigma(X,F)$-closed subspace $Y$ of $X$. Following~\cite{DKLR16}, we will consider the set $\mathcal C$ of finite codimensional $\sigma(X,F)$-closed subspaces of $X$ as a directed set with the order $\preceq$ given by $E\preceq F$ if $F\subset E$. 

\begin{proposition}\label{prop:seqMod} Let $F$ be a norming subspace of $X^*$. For each $x\in S_X$ and $t>0$ we have:
\begin{align*} 
\d_X^F(t,x) &= \inf_{\substack{x_\alpha\stackrel{\sigma(X,F)}{\longrightarrow}0\\ \norm{x_\alpha}\geq t}} \liminf_{\alpha} \norm{x+x_\alpha}-1;\\
\p_X^F(t,x) &= \sup_{\substack{x_\alpha\stackrel{\sigma(X,F)}{\longrightarrow}0\\ \norm{x_\alpha}\leq t}} \limsup_{\alpha} \norm{x+x_\alpha}-1\,.
\end{align*}
If moreover $F$ is separable then
\begin{align*} 
\d_X^F(t,x) &= \inf_{\substack{x_n\stackrel{\sigma(X,F)}{\longrightarrow}0\\ \norm{x_n}\geq t}} \liminf_{n\to\infty} \norm{x+x_n}-1;\\
\p_X^F(t,x) &= \sup_{\substack{x_n\stackrel{\sigma(X,F)}{\longrightarrow}0\\ \norm{x_n}\leq t}} \limsup_{n\to\infty} \norm{x+x_n}-1\,.
\end{align*}
\end{proposition}

\begin{proof}
We will prove the first formula for $\d_X^F(t,x)$, since the proof of the one for $\p_X^F(t,x)$ is similar. Let us consider 
\[ \theta(t,x)=\inf_{\substack{x_\alpha\stackrel{\sigma(X,F)}{\longrightarrow}0\\\norm{x_\alpha}\geq t}} \liminf_\alpha \norm{x+x_\alpha}-1\,.\]
Fix $\varepsilon>0$. For each finite codimensional $\sigma(X,F)$-closed subspace $Z$ of $X$, take  $x_Z\in S_Z$ so that $\norm{x+tx_Z}\leq \inf_{y\in S_Z} \norm{x+ty}+\varepsilon$. Note that the net $(x_Z)_{Z\in \mathcal C}$ is $\sigma(X,F)$-convergent to $0$. Indeed, given $f\in F$ we have that $f(x_Z)=0$ whenever $Z\subset \ker f$. Thus, 
\[ \theta(t,x)\leq \liminf_{Z\in\mathcal C} \norm{x+tx_Z}-1 \leq \d_X^F(t,x)+\varepsilon\,.\]
Letting $\varepsilon\to 0$, we get $\theta(t,x)\leq \d_X^F(t,x)$. Now, take $(x_\alpha)_\alpha$ a $\sigma(X,F)$-null net such that $\norm{x_\alpha}\geq t$ for each $\alpha$. Fix $\varepsilon>0$ and take $Y$ a finite codimensional $\sigma(X,F)$-closed subspace of $X$. Then $\lim_\alpha d(x_\alpha, Y)= 0$, so there exists a net $(y_\alpha)_\alpha$ in $Y$ and $\alpha_0$ so that if $\alpha\geq\alpha_0$ then $\norm{x_\alpha-y_\alpha}\leq \varepsilon$. Thus $\norm{y_\alpha}\geq t-\varepsilon$ whenever $\alpha\geq \alpha_0$. Moreover, from the convexity of the function $t\mapsto \norm{x+ty_\alpha}-1$ we get that
\[ \norm{x+x_\alpha}-1 \geq \norm{x+y_\alpha}-1-\varepsilon 
\geq \frac{\norm{y_\alpha}}{t-\varepsilon}\left(\norm{x+\frac{t-\varepsilon}{\norm{y_\alpha}} y_\alpha}-1\right)-\varepsilon.
\]
It follows that $\liminf_\alpha \norm{x+x_\alpha}-1 \geq \d_X^F(t-\varepsilon,x,Y)-\varepsilon$. That inequality holds for every finite codimensional $\sigma(X,F)$-closed subspace $Y$ of $X$ and every $\varepsilon>0$. Since the function $t\mapsto \d_X^F(t,x)$ is $1$-Lipschitz, we get
\[ \liminf_\alpha \norm{x+x_\alpha}-1 \geq \d_X^F(t,x)\,,\]
as desired. 

Finally, assume that $F$ is separable. From what we have already proved it follows
\begin{align*} 
\d_X^F(t,x) &\leq \inf_{\substack{x_n\stackrel{\sigma(X,F)}{\longrightarrow}0\\ \norm{x_n}\geq t}} \liminf_{n\to\infty} \norm{x+x_n}-1\,,\\
\p_X^F(t,x) &\geq \sup_{\substack{x_n\stackrel{\sigma(X,F)}{\longrightarrow}0\\ \norm{x_n}\leq t}} \limsup_{n\to\infty} \norm{x+x_n}-1\,.
\end{align*}
Let $\{f_n:n\in\N\}$ be a dense sequence in $F$. Let us consider the finite codimensional $\sigma(X,F)$-closed subspaces of $X$ given by $Y_n = \bigcap_{i=1}^n \ker f_i$, for each $n\in \N$. Fix $0<\varepsilon<t$. For every $n$, take $x_n, y_n\in Y_n$ such that $\norm{x_n}=\norm{y_n}= 1$ and
\begin{align*}
 \norm{x+tx_n} &\leq \inf_{y\in S_{Y_n}} \norm{x+ty}+ \varepsilon \leq \d_X^F(x,t)+1+\varepsilon\,, \\
 \norm{x+ty_n} &\geq \sup_{y\in S_{Y_n}} \norm{x+ty}- \varepsilon \geq \p_X^F(x,t)+1-\varepsilon\,.
 \end{align*}
It is easy to check that the sequences $(x_n)_n$ and $(y_n)_n$ are $\sigma(X,F)$-null. Since $\varepsilon$ was arbitrary, this finishes the proof.
\end{proof}

For the norming subspaces that we will consider in the next sections, the norm will be $\sigma(X,F)$-lower semicontinuous. In view of Proposition~\ref{prop:seqMod}, that condition guarantee that $\d_X^F$ and $\p_X^F$ are non-negative functions. 

It is easy to show that, if $\E$ is an FDD for $X$, then $F = \cvspan\{(P_n^\E)^* X^*: n\in \N\}$ is a norming subspace of $X^*$. Moreover, if $\E$ is monotone then $F$ is $1$-norming and $\norm{\cdot}$ is $\sigma(X,F)$-lower semicontinuous. 

\begin{proposition}\label{prop:FDDmod}  Let $\E$ be a monotone FDD for a Banach space $X$ and $F = \cvspan\{(P_n^\E)^* X^* : n\in \N\}$. For each $t>0$ we have:
\begin{align*}
\d_X^{F}(t) &= \inf_{n\in\N}\sup_{m\geq n}\inf\{\norm{x+ty}-1: x\in H_n\cap S_X, y\in H^m\cap S_X\}\,,\\
\p_X^{F}(t) &= \sup_{n\in\N}\inf_{m\geq n}\sup\{\norm{x+ty}-1: x\in H_n\cap S_X, y\in H^m\cap S_X\}
\end{align*}
where $H_n = \bigoplus_{i=1}^n E_i$ and $H^n = \closure{\bigoplus_{i=n+1}^\infty E_i}$ for each $n\in \N$. 
\end{proposition}

\begin{proof}
 We prove only the statement concerning $\p_X^{F}$ as the other one is similar. Since $\cup_n H_n \cap S_X$ is dense in $S_X$, we have that $\p_X^{F}(t) = \sup_{n\in\N}\sup_{x\in H_n\cap S_X}\p_X^{F}(t,x)$ (see~\cite[Lemma 1]{GJT07}). Thus, it suffices to show that 
 \[\sup_{x\in H_n\cap S_X}\p_X^{F}(t,x) = \inf_{m\geq n}\sup\{\norm{x+ty}-1: x\in H_n\cap S_X, y\in H^m\cap S_X\}\]
 for each $n\in \N$. First notice that each $H^m$ is a $\sigma(X,F)$-closed subspace of $X$. Indeed, since $H^m= \ker P_m^\E$, it suffices to check that $P_m^\E$ is $\sigma(X,F)$-continuous for every $m\in \mathbb N$. For that, take a net $(x_\alpha)_\alpha$ that is $\sigma(X,F)$-converging to a vector $x\in X$. Then $\lim_\alpha (P_n^\E)^*(x^*)(x_\alpha)= (P_n^\E)^*(x^*)(x)$ for each $x^*\in X^*$, so $(P_n^\E x_\alpha)_\alpha$ is $\sigma(X,X^*)$-convergent to $P_n^\E(x)$. Since $P_n^\E(X)$ is finite-dimensional, it follows that $(P_n^\E x_\alpha)_\alpha$ is also norm-convergent. This shows that $P_m^\E$ is $\sigma(X,F)$-continuous and so $H^m$ is $\sigma(X,F)$-closed. Therefore, 
  \[\sup_{x\in H_n\cap S_X}\p_X^{F}(t,x) \leq \inf_{m\geq n}\sup\{\norm{x+ty}-1: x\in H_n\cap S_X, y\in H^m\cap S_X\}\]
Now, fix $n\in\N$. Assume that 
\[  \sup_{x\in H_n\cap S_X}\p_X^{F}(t,x) < \rho < \rho+\varepsilon < \inf_{m\geq n}\sup\{\norm{x+ty}-1: x\in H_n\cap S_X, y\in H^m\cap S_X\}\]
 for some $\rho,\varepsilon>0$. We claim that for each $x\in H_n$ there exists $m=m(x)>n$ such that $\norm{x+ty} < 1+\rho$ for each $y\in H^m\cap S_X$. To see this, assume that there exist $x\in H_n$ and a sequence $(y_m)_m$ so that $y_m\in H^m\cap S_X$ and $\norm{x+ty_m}\geq 1+\rho$ whenever $m\geq n$. Note that $F$ is separable and the sequence $(y_m)_m$ is $\sigma(X,F)$-null. Therefore, the sequential formula for the modulus given in Proposition~\ref{prop:seqMod} yields 
\[\rho\leq\limsup_{m\to\infty}\norm{x+ty_m}-1\leq \p_X^F(t,x)<\rho\,,\]
which is a contradiction. This proves the claim. 
 Now pick $\{x_i\}_{i=1}^k$ an $\varepsilon$-net in $H_n\cap S_X$, take $m=\max\{m(x_i):i=1\ldots,k\}$ and let $x\in H_n$ and $y\in H^m$ be norm-one vectors. There exists $i$ such that $\norm{x-x_i}\leq \varepsilon$. Then,
 \[ \norm{x+ty}-1\leq \norm{x_i+ty}-1+\varepsilon \leq \rho+\varepsilon\,,\]
 which is a contradiction.
\end{proof}

Let us recall that if $\E=(E_n)_n$ is a monotone FDD in $X$ with associated projections $(P_n^\E)_n$, then $\Ed = ((P_n^\E - P_{n-1}^\E)^*X^*)_n$ is an FDD for $F=\cvspan\{(P_n^{\E})^* X^*: n\in\N\}$ with associated projections given by $P_n^\Ed = (P_n^\E)^*$. Note that if $\E$ is shrinking then $F=X^*$ and $\Ed$ is boundedly complete. Proposition~\ref{prop:FDDmod} provides a formula for the asymptotic moduli in spaces admitting a monotone shrinking FDD.

\begin{corollary}\label{cor:FDDmod}  Let $X$ be a Banach space admitting a monotone shrinking FDD $\E$. For each $t>0$ we have:
\begin{align*}
\d_X(t) &= \inf_{n\in\N}\sup_{m\geq n}\inf\{\norm{x+ty}-1: x\in P_n^\E(X) \cap S_X, y\in \ker P_m^\E\cap S_X\}\,,\\
\p_X(t) &= \sup_{n\in\N}\inf_{m\geq n}\sup\{\norm{x+ty}-1: x\in P_n^\E(X) \cap S_X, y\in \ker P_m^\E \cap S_X\}\,,\\
\d^*_X(t) &= \inf_{n\in\N}\sup_{m\geq n}\inf\{\norm{x^*+ty^*}-1: x^*\in (P_n^{\Ed}X^*) \cap S_X, y^*\in \ker P_m^{\Ed}\cap S_X\}\,.
\end{align*}
\end{corollary}

\section{On strongly AUC and strongly AUS spaces}\label{sec:sauc}

The following definition is motivated by the formulae obtained in Proposition~\ref{prop:FDDmod} and Corollary~\ref{cor:FDDmod}. 

\begin{definition} Let $X$ a Banach space and let $\E=(E_n)_n$ be an FDD for $X$. Denote $H_n = \bigoplus_{i=1}^n E_i$ and $H^n = \closure{\bigoplus_{i=n+1}^\infty E_i}$. The space $X$ is said to be \emph{strongly AUC with respect to $\E$} if the modulus defined by
\[ \sd{\E}(t) =  \inf_{n\in\N}\sup_{m\geq n}\inf\{\norm{x+ty}-1: x\in H_m\cap S_X, y\in H^m\cap S_X\} \]
satisfies that $\sd{\E}(t) > 0$ for each $t>0$. In addition, $X$ is said to be \emph{strongly AUS with respect to $\E$} if 
\[ \sp{\E}(t) =  \sup_{n\in\N}\inf_{m\geq n}\sup\{\norm{x+ty}-1: x\in H_m\cap S_X, y\in H^m\cap S_X\}\]
satisfies $\lim_{t\to 0} t^{-1} \sp{\E}(t)=0$. 
Finally, we say that $X$ is \emph{strongly AUS} (resp.\ \emph{strongly AUC}) if $X$ is strongly AUS (resp.\ strongly AUC) with respect to some FDD. 
\end{definition}

Since $\max\{\norm{x+y},\norm{x-y}\}\geq \norm{x}$ for each $x,y\in X$, it follows that $\sp{\E}(t)\geq 0$ for each $t$. Moreover, if $\E$ is monotone then $\sd{\E}(t)\geq 0$. It is clear that functions $\sp{\E}$ and $\sd{\E}$ are $1$-Lipschitz functions and $\sd{\E}(t) \leq \sp{\E}(t)\leq t$ for all $t$. For notational convenience let us set
\begin{align*}
\sd{\E}(t,m) &= \inf\{\norm{x+ty}-1: x\in H_m\cap S_X, y\in H^m\cap S_X\}\,,\\ 
\sp{\E}(t,m) &=  \sup\{\norm{x+ty}-1: x\in H_m\cap S_X, y\in H^m\cap S_X\}\,.
\end{align*}
Note that if $\F$ is a blocking of $\E$ then for each $m$ there is $k_m\geq m$ so that $\sd{\F}(t,m)=\sd{\E}(t,k_m)$ and $\sp{\F}(t,m)=\sp{\E}(t,k_m)$. Thus, $\sd{\F}(t) \leq \sd{\E}(t)$ and $\sp{\F}(t)\geq \sp{\E}(t)$. In particular, $X$ is strongly AUC (resp.\ strongly AUS) with respect to $\E$ whenever it is strongly AUC (resp.\ strongly AUS) with respect to some blocking of $\E$.

Remark that, in the above definitions, we only compute the norms $||x+ty||$ for vectors $x$ and $y$ which belong to \emph{complementary} subspaces, that is, $x\in H_m$ and $y\in H^m$ for a certain $m$. This is why we called this notions strong AUS and strong AUC. Indeed, 
as a consequence of Corollary~\ref{cor:FDDmod} we obtain the following:

\begin{corollary} \label{cor:SAUSimplyAUS} Let $\E$ be a monotone shrinking FDD for a Banach space $X$. Then $\d_X(t) \geq \sd{\E}(t)$, $\p_X(t)\leq \sp{\E}(t)$ and $\d_{X}^* (t)\leq \sd{\E^*}(t)$. Thus, $X$ is AUC (resp. AUS) whenever it is strongly AUC (resp. strongly AUS) with respect to $\E$, and $X^*$ is weak* AUC whenever it is strongly AUC with respect to $\E^*$. 
\end{corollary}

\begin{example}
\begin{enumerate}[leftmargin=*]
\item[a)] Let $X=(\bigoplus_{n=1}^\infty E_n)_p$ be an $\ell_p$-sum of finite dimensional spaces, $1\leq p<\infty$, and consider $\E = (E_n)_{n=1}^\infty$. Then $\sd{\E}(t)=\sp{\E}(t)=(1+t^p)^{1/p}-1$. Thus, $X$ is strongly AUC with respect to $\E$. If moreover $p>1$ then it is strongly AUS with respect to $\E$. 
\item[b)] Let $X=(\bigoplus_{n=1}^\infty E_n)_0$ be a $c_0$-sum of finite dimensional spaces, and $\E = (E_n)_{n=1}^\infty$. Then $\sp{\E}(t)=0$ for each $t\in(0,1]$, so $X$ is strongly AUS with respect to $\E$.
\item[c)] Consider the James space $J$ endowed with the norm
\[ \norm{(x_n)_{n=1}^\infty}^2 = \sup_{1\leq n_1<\ldots<n_{2m+1}} \sum_{i=1}^m(x_{n_{2i-1}}-x_{n_{2i}})^2+ 2x_{n_{2m+1}}^2 \]
given in~\cite{P99} and let $\E$ be the standard basis of $J$. Then $\norm{x+y}^2\leq \norm{x}^2+2\norm{y}^2$ whenever $x\in H_n$ and $y\in H^n$ for some $n$. Thus, $\sp{\E}(t)\leq (1+2t^2)^{1/2}-1$, so $J$ is strongly AUS with respect to $\E$.
\item[d)] Let $T$ be a well-founded tree in $\omega^{<\omega}$. The James Tree space $JT$ consists of all real functions defined on $T$, with the norm
\[ \norm{x}^2 = \sup\sum_{j=1}^n\left(\sum_{t\in S_j} x(t)\right)^2\]
where the supremum is taken over all finite sets of pairwise disjoint segments in $T$. Lancien proved in~\cite[Proposition 4.6]{L92} that there exists a basis $\E=(e_n)_n$ of $JT$ and an increasing sequence $(n_k)_k$ such that if $x\in \vspan\{e_1,\ldots, e_{n_k}\}$ and $y\in \cvspan\{e_i : i>n_k\}$ then $\norm{x+y}^2\geq \norm{x}^2+\norm{y}^2$. Therefore $JT$ is strongly AUC with respect to $\E$ and $\sd{\E}(t)\geq (1+t^2)^{1/2}-1$.
\end{enumerate}
\end{example} 

Recall that the \emph{modulus of convexity} of a Banach space $X$ is defined by
\[ \delta_X(t) = \inf\{1-\norm{\frac{x+y}{2}} : x,y\in S_X, \norm{x-y}=t\}\,,\]
and the \emph{modulus of smoothness} of $X$ is defined by
\[\rho_X(t) = \frac{1}{2}\sup\{\norm{x+ty}+\norm{x-ty}-2 : x,y\in S_X\}\,.\]

\begin{proposition}\label{prop:UCimplySAUC} Let $\E$ be a monotone FDD for a Banach space $X$. Then $\delta_X(t)\leq\sd{\E}(t)$ and $\sp{\E}(t)\leq 2\rho_X(t)$ for each $t>0$. Thus, if $X$ is uniformly convex (resp. uniformly smooth) then it is strongly AUC (resp. strongly AUS) with respect to $\E$. 
\end{proposition}

\begin{proof}
We will use the same arguments that appear in Proposition 2.3.(3) in~\cite{JLPS02}. From the monotony of $\E$ follows that 
\[ \frac{1}{2}(\norm{x+ty}-1) \leq \frac{1}{2}(\norm{x+ty}+\norm{x-ty})-1 \]
whenever $x\in H_n\cap S_X$ and $y\in H^n\cap S_X$ for some $n\in \N$. Thus, $\sp{\E}(t)\leq 2\rho_X(t)$.
Now, fix $n\in\N$ and take $x\in H_n\cap S_X$ and $y\in H^n\cap S_X$. Let $x^*\in S_{X^*}$ be such that $x^*(x)=1$. Then $y^* = x^*\circ P_n^*$ satisfies $\norm{y^*}= 1$, $y^*(x)=1$ and $y^*(y)=0$. Let us consider $u= \frac{x+ty}{\norm{x+ty}}$ and $v=u-ty$. Then $u,v\in B_X$ and $\norm{u-v}=t$. Thus,
\[ \delta_X(t) \leq 1 - \frac{1}{2}\norm{u+v}\leq 1-\frac{1}{2}y^*(u+v) = 1-\frac{1}{\norm{x+ty}}\leq \norm{x+ty}-1 \]
and so $\delta_X(t) \leq \sd{E}(t)$.
\end{proof}

Our next result establishes the duality between strongly AUS and strongly AUC norms by using estimates similar to those in~\cite{DKLR16}. Recall that, given a continuous function $f:[0,1]\to[0,+\infty)$ with $f(0)=0$, its \emph{dual Young function} is defined by
\[ f^*(s)=\sup\{st-f(t) : 0\leq t\leq 1 \}\,. \]

\begin{proposition}\label{prop:SAUSduality} Let $\E=(E_n)_n$ be a monotone FDD for a Banach space $X$ and let $\Ed$ be the dual FDD for $F=\cvspan\{(P_n^{\E})^* X^*:n\in \N\}$ given above. Take $0<s,t<1$. Then:
\begin{itemize}
\item[a)] If $\sp{\E}(s)<st$, then $\sd{\Ed}(3t)\geq st$.
\item[b)] If $\sd{\Ed}(t)>st$, then $\sp{\E}(s)\leq st$. 
\end{itemize}
Therefore, $\sp{\E}^*$ is equivalent to $\sd{\Ed}$ and $\sp{\Ed}$ is equivalent to $\sd{\E}^*$.
\end{proposition}

\begin{proof}
As usual, let us consider $H_n = \bigoplus_{i=1}^n E_i$ and $H^n = \closure{\bigoplus_{i=n+1}^\infty E_i}$. In order to prove $a)$, assume that $\sp{\E}(s)<st$ and fix $\varepsilon>0$ and $n\in \N$. Take $m\geq n$ so that $\sp{\E}(s,m) < st$. Let $f\in P_m^\Ed (X^*)\cap S_X^*$ and $g\in \ker P_n^\Ed\cap S_{X^*}$. We will estimate $\norm{f+3tg}$. Note that, by the monotony of $\E$, there exists $x\in H_m\cap S_X$ such that $f(x)>1-\varepsilon$. Now take $y\in H^m\cap S_X$. We have $\norm{x+s y}<1+st$. Thus,
\begin{eqnarray*}
\norm{f+3t g}&\geq & \frac{1}{1+st}(f+3t g)(x+s y) \\
&=& \frac{1}{1+st}(f(x)+3st g(y))\geq \frac{1-\varepsilon+3st g(y)}{1+st}
\end{eqnarray*}
From the monotony of $\E$, it follows that $\norm{g}=\sup\{g(y): y\in H^n\cap S_X\}$. Thus,
\[ \norm{f+3t g} \geq \frac{1-\varepsilon+3st}{1+st}\]
Hence,
\[ \sd{\Ed}(3t,m) \geq \frac{1-\varepsilon+3st}{1+st}-1\,. \]
Since $\varepsilon$ is arbitrary, we get that for every $n$ there exists $m\geq n$ so that $\sd{\Ed}(3t,m)\geq st$. Therefore $\sd{\Ed}(3t)\geq st$, as desired. 
 
Now we turn to the proof of $b)$. Assume that $\sd{\Ed}(t)> st$ and $\sp{\E}(s)>st$. Then there exist $\rho>st$ and $n\in \N$ such that $\inf_{m\geq n} \sp{\E}(t,m) > \rho$. Moreover, there is $m\geq n$ so that $\sd{\Ed}(t,m)>st$. Now take $x\in H_m\cap S_X$ and $y\in H^m\cap S_X$ satisfying 
\[ 1+st < 1+\rho < \norm{x+s y} \,.\]
Let $z^*\in S_{X^*}$ be such that $z^*(x+s y)=\norm{x+s y}$. Take $f=P_m^\Ed z^*$, $g=(I-P_m^\Ed) z^*$ and $c=\norm{g}$. Since $\E$ is monotone, we get that
\[ 1+st< z^*(x+s y) = f(x)+g(s y)\leq 1 + c s\,. \]
Thus, $t<c$. We claim that $\norm{f}\leq 1-c s$. Since $\sd{\Ed}(t,m)> st$, we get that
\[
(1+st)\norm{f} < \norm{ f+\frac{t}{c}\norm{f}g} 
\leq  \frac{t}{c}\norm{f} \norm{f+g} + (1-\frac{t}{c}\norm{f})\norm{f}\,. \\
\]
Hence, $1+st< \frac{t}{c} + (1-\frac{t}{c}\norm{f})$. That proves the claim. Now, 
\[ 1+st < z^*(x+sy) = f(x)+g(sy) \leq 1-cs +cs = 1\,,\]
which is a contradiction. 

Finally, a standard argument shows from what we have already proved that $\sd{\Ed}(t/2) \leq \sp{\E}^*(t) \leq \sd{\Ed}(3t)$, so $\sp{\E}^*$ is equivalent to $\sd{\Ed}$. On the other hand, it is easy to check that if $P$ is a norm-one projection on $X$ with finite-dimensional range, then $P^{**}(X^{**})$ is isometric to $P(X)$. Thus, $X=\cvspan\{(P_n^{\Ed})^*(X^{**})\}$ and $\E^{**}$ may be identified with $\E$. By applying the previous formula to $\Ed$ we get that $\sd{\E}$ is equivalent to $\sp{\Ed}^*$, which finishes the proof. 
\end{proof}

\begin{corollary} \label{cor:SAUSduality} Let $X$ be a Banach space with a monotone shrinking FDD $\E$. Then $X$ is strongly AUS (resp. strongly AUC) with respect to $\E$ with power type $p$ if and only if $X^*$ is strongly AUC (resp. strongly AUS) with respect to the dual FDD $\Ed$ with power type $p'$, the conjugate exponent of $p$.
\end{corollary}

Given an FDD $\E$ for $X$, an element $x\in X$ is said to be a \emph{block} of $\E$ if $x = P_n^\E x$ for some $n$. The interval
\[ \ran_\E x = [\max\{n: P_n^\E x = 0\}+1, \min\{n : P_n^\E x=x\}]\]
is called the \emph{range} of the block $x$. Given $1\leq p,q\leq \infty$, it is said that $\E$ \emph{satisfies $(p,q)$-estimates} if there exists a constant $C>0$ such that
\[ \frac{1}{C}\left(\sum_{i=1}^n \norm{x_i}^p\right)^{1/p} \leq \norm{\sum_{i=1}^n x_i} \leq C\left(\sum_{i=1}^n \norm{x_i}^q\right)^{1/q} \]
for all finite sequences $x_1,\ldots, x_n$ with $\ran_\F x_i \cap \ran_\F x_j = \emptyset$ for every $i\neq j$.
 
The next result is based on a similar one given by Prus in~\cite{P89} for NUS spaces.

\begin{proposition} \label{prop:AUSestimate} Let $\E$ be an FDD for a Banach space $X$. 
\begin{itemize}
\item[a)] If $X$ is strongly AUS with respect to $\E$ then there is a blocking $\F=(F_n)_n$ of $\E$ satisfying $(\infty,q)$-estimates for some $1<q<\infty$.  
\item[b)] If $\E$ is monotone and $X$ is strongly AUC with respect to $\E$ then there is a blocking $\F=(F_n)_n$ of $\E$ satisfying $(1,p)$-estimates for some $1<p<\infty$. 
\end{itemize}
\end{proposition}

\begin{proof} We will mimic the proof of~\cite[Theorem 3.3]{P89}. First assume that $X$ is strongly AUS with respect to $\E$ and fix $t>0$ such that $\sp{\E}(t)<t/2$. Thus, there exists an increasing sequence $(m_n)_n\subset \mathbb N$ so that $m_1=0$ and $\sp{\E}(t,m_n)<t/2$ for $n> 1$. Consider $F_n = \bigoplus_{i=m_n+1}^{m_{n+1}} E_i$ and let $q>1$ be such that $(2-t/2)^q<2$. Take $\nu<1/2$ so that $(1+\alpha-t/2)^q<1+\alpha^q$ whenever $|1-\alpha|<\nu$. Note that for such $\alpha$, if $x\in \bigoplus_{i=1}^n F_i\cap S_X$ and $y=\closure{\bigoplus_{i=n+1}^\infty F_i}\cap S_X$ for some $n$, then
\[
 \norm{x+\alpha y} \leq \norm{x+ty}+(\alpha-t)\norm{y} \leq 1+\alpha-\frac{t}{2} \leq (1+\alpha^q)^{1/q}\,.
\]
Now one can follow the same steps as in the proof of Gurarii's theorem (see, e.g.~\cite[Lemma 9.26]{FHHMZ11}) to get the statement. 

On the other hand, assume that $\E$ is monotone and $X$ is strongly AUC with respect to $\E$. We will argue as in~\cite[Lemma 9.27]{FHHMZ11}. By Proposition~\ref{prop:SAUSduality}, $F=\cvspan\{(P_n^\E)^* X^*:n\in \N\}$ is strongly AUS with respect to $\Ed$. From what we have already proved we get $q>1$, $C>0$ and an increasing sequence $(m_n)_n$ so that the FDD $\F=(F_n)_n$ given by $F_n = \bigoplus_{i=m_n+1}^{m_{n+1}} (P_{n+1}^\E-P_n^\E)^* X^*$ is a blocking of $\Ed$ which satisfies $(\infty,q)$-estimates with constant $C$. Now, take $p = \frac{q-1}{q}$ and $\G=(G_n)_n$ given by $G_n = \bigoplus_{i=m_n+1}^{m_{n+1}} E_i$. We will show that $p$ and $\G$ do the work. For that, let $x_1,\ldots, x_n\in X$ with $\ran_{\G}x_i \cap \ran_{\G}x_j=\emptyset$ for all $i\neq j$. For each $i$, take $f_i \in F$ such that $\norm{f_i}=1$ and $f_i(x_i)=\norm{x_i}$, which exists since $\E$ is monotone and so $F$ is a $1$-norming subspace of $X^*$. Moreover, we may replace $f_i$ by $f_i \circ (P_{\max\ran_{\G}(x_i)}^\G- P_{\min\ran_{\G}(x_i)-1}^\G)$ to get that $\ran_{\F}(f_i)=\ran_{\G}(x_i)$ for each $i$. Thus, $f_1,\ldots, f_n$ have pairwise disjoint ranges and $f_i(x_j)=\delta_{ij} \norm{x_i}$ for each $i,j$. Now let $f=\sum_{i=1}^n \beta_i f_i$, where $\beta_i = \norm{x_i}^{1/(q-1)}$. Then
\[
\norm{\sum_{i=1}^n x_i} \geq \frac{1}{\norm{f}} f\left(\sum_{i=1}^n x_i\right) \geq \frac{ \sum_{i=1}^n \beta_i \norm{x_i}} {C\left(\sum_{i=1}^n \norm{\beta_i f_i}^q\right)^{\frac{1}{q}}} = \frac{1}{C}\left(\sum_{i=1}^n \norm{x_i}^p\right)^\frac{1}{p}\,,
\]
as desired. 
\end{proof}

Recall that an FDD $\E=(E_n)_n$ is said to be \emph{unconditional} if there exists a constant $L>0$ so that for every $n$ and every $A\subset\{1,\ldots,n\}$ we have $||\sum_{i\in A}x_i||\leq L||\sum_{i=1}^n x_i||$ whenever $x_i\in E_i$ for each $i=1,\dots,n$. 

It is well-known that every FDD satisfying $(\infty,q)$-estimates for some $q>1$ is shrinking, and every FDD satisfying $(p,1)$-estimates for some $p<\infty$ is boundedly complete. Moreover, an FDD is shrinking (resp.\ boundedly complete) if it has an shrinking (resp.\ boundedly complete) blocking. This yields to the following result.

\begin{proposition}\label{prop:AUSshrinking} Let $\E$ be an FDD for a Banach space $X$.
\begin{itemize}
\item[a)] If $X$ is strongly AUS with respect to $\E$ then $\E$ is shrinking. 
\item[b)] If $\E$ is either unconditional or monotone and $X$ is strongly AUC with respect to $\E$ then $\E$ is boundedly complete.
\end{itemize}
Thus, $X$ is reflexive whenever it is both strongly AUS and strongly AUC with respect to some unconditional FDD.
\end{proposition}

\begin{proof}
Both a) and b) in the monotone case follow from Proposition~\ref{prop:AUSestimate}. Now assume that $\E$ is unconditional with constant $L>0$. Assume that $X$ is strongly AUC with respect to $\E$ and $\E$ is not boundedly complete. Then there exists $0<M<1$, an increasing sequence $(k_n)_n\subset \mathbb N$ and a sequence $(u_n)_n$ such that $u_n\in\bigoplus_{i=k_{n}+1}^{k_{n+1}} E_i$ such that $\norm{u_n}\geq M$ and $\norm{\sum_{i\leq n} u_i}\leq 1$ for each $n$. We may inductively pick increasing sequences $(n_j)_j$ and $(m_j)_j$ satisfying that $k_{n_{j-1}+1}\leq m_j<k_{n_{j}}$ and  $ \sd{\E}(t,m_j) \geq \sd{\E}(t)(1- 2^{-j-1})>0$. 
We claim that 
\begin{equation}\label{claim:bddly}
L \geq \norm{\sum_{i\leq j} u_{n_i}} \geq M + (j-1) \sd{\E}(M)(1-\sum_{i=1}^j 2^{-i-1})
\end{equation}
for each $j\geq 1$. Indeed, for $j=1$ the statement is clear. Moreover, assume that the claim holds for $j-1$ and take $x=\sum_{i\leq j-1} u_{n_i}$. Then the convexity of the function $t\mapsto \norm{\frac{x}{\norm{x}}+t\frac{u_{n_j}}{\norm{u_{n_j}}}}$ and the fact that $x\in \bigoplus_{i=1}^{m_j} E_i$ and $u_{n_j}\in \overline{\bigoplus_{i=m_j+1}^\infty E_i}$ imply 
\allowdisplaybreaks
\begin{align*}
L  &\geq \norm{\sum_{i\leq j} u_{n_i}} = \norm{x}\norm{\frac{x}{\norm{x}}+\frac{u_{n_j}}{\norm{x}}} \\
&\geq \norm{x}\left(1+\frac{\norm{u_{n_j}}}{M\norm{x}}\left(\norm{\frac{x}{\norm{x}}+M\frac{u_{n_j}}{\norm{u_{n_j}}}}-1\right)\right)\\
&\geq \norm{x} + \sd{\E}(M)(1-2^{-j-1})\\
&\geq M + (j-2)\sd{\E}(M)(1-\sum_{i=1}^{j-1} 2^{-i+1}) + \sd{\E}(M)(1-2^{-j-1})\\
&\geq M + (j-1) \sd{\E}(M)(1-\sum_{i=1}^j 2^{-i-1})\,.
\end{align*}
That proves the claim. Finally, from (\ref{claim:bddly}) follows that $L\geq M + 2^{-1}(j-1)\sd{\E}(M)$ for each $j\geq 1$. Thus, $\sd{\E}(M)\leq 0$, which is a contradiction. 
\end{proof}
From the point of view of renorming theory, the strong asymptotic properties introduced above turn out to be equivalent to the classical ones on reflexive spaces admitting a FDD. This follows from Prus' characterisation of NUC norms~\cite{P89}. 

\begin{proposition} Let $\E$ be an FDD for a Banach space $X$. If $X$ is AUC (respectively, AUS), then there is an equivalent norm $\eqnorm$ in $X$ and a blocking $\F$ of $\E$ such that $(X,\eqnorm)$ is strongly AUC (respectively, strongly AUS) with respect to $\F$.  
\end{proposition}

\begin{proof}
First assume that $X$ is AUC. Since $X$ is reflexive, it is also NUC. Theorem 4.2 in~\cite{P89} provides a blocking $\F=(F_n)_n$ of $\E$ which satisfies $(p,1)$-estimates with constant $C>0$ for some $p>1$ (actually, Theorem 4.2 in~\cite{P89} is stated for $\E$ being a basis, but it also works for FDDs). Following~\cite{P89}, given a block $x\in X$ we define
\[ |||x|||^p:= \sup\left\{\sum_{i=1}^n ||x_i||^p : x= \sum_{i=1}^n x_i, \ran_\F x_i \cap \ran_\F x_j = \emptyset \text{ for all } i\neq j\right\}.\]
Then $\eqnorm$ can be extended to a norm in $X$ which satisfies $||x||\leq |||x|||\leq C^{-1}||x||$ for every $x\in X$. 
Moreover, $|||x|||^p+|||y|||^p \leq |||x+y|||^p$ whenever $x\in \bigoplus_{i=1}^n F_i$ and $y\in \closure{\bigoplus_{i=n+1}^\infty F_i}$. Therefore $(X,\eqnorm)$ is strongly AUC with respect to $\F$ with modulus $\sd{\F}(t)\leq (1+t^p)^{1/p}-1$, as desired. 

Finally, assume $X$ is AUS. Then $X^*$ is AUC and so there is a blocking $\F$ of $\Ed$ and an equivalent norm in $X^*$ such that $X^*$ is strongly AUC with respect to $\F$ under this new norm. Now the result follows from the duality between strongly AUC and strongly AUS norms proved in Proposition~\ref{prop:SAUSduality}. 
\end{proof}

We finish the section by providing some examples of spaces having a basis which satisfy the classical asymptotic properties but not the stronger ones. 

\begin{example} 
\begin{enumerate}[leftmargin=*]
\item[a)] Johnson and Schechtman constructed in~\cite{JO01} a subspace $Y$ of $c_0$ with a basis such that $Y^*$ does not have the approximation property. Thus, $Y$ is an AUS space and it does not admit a shrinking FDD. Therefore $Y$ is not a strongly AUS space. 
\item[b)] Girardi proved in~\cite{G01} that $JT_*$, the predual of the James Tree space, is an AUC space. Since $JT_*$ is not isomorphic to a dual space, it does not admit a boundedly complete FDD. Thus $JT_*$ is not strongly AUC space with respect to any either unconditional or monotone FDD. 
\end{enumerate}
\end{example}

Note that the failure of strong asymptotic properties in previous examples relies on the lack of reflexivity. We do not know any example of reflexive Banach space admitting an FDD with is AUC but not strongly AUC.

\section{Asymptotically uniformly smooth injective tensor products}\label{sec:iten}

Let us recall that if $T:X\to X$ and $R: Y\to Y$ are linear operators then 
\[ (T\otimes S)(\sum_{i=1}^n x_i\otimes y_i ) = \sum_{i=1}^n T(x_i)\otimes S(y_i)\]
defines a linear operator from $X\iten Y$ to $X\iten Y$ such that $\norm{T\otimes S} = \norm{T} \cdot \norm{S}$. 

\begin{theorem}\label{th:AUSiten} Let $\E, \F$ be FDDs on Banach spaces $X$ and $Y$, respectively. Then there exists a constant $K>0$ such that 
\[ \p_{X\iten Y}(t) \leq (1+\sp{\E}(Kt))(1+\sp{\F}(Kt))-1 \]
for every $0<t\leq 1/K$.
\end{theorem}

\begin{proof}
Let $Q_n^E=I-P_n^E$ and $Q_n^\F = I-P_n^\F$ be the complementary projections. Take $K=\sup\{\norm{P_n^\E},\norm{P_n^\F}:n\in \mathbb N\}$.
Fix $0<t\leq \frac{1}{4K^2}$ and $\varepsilon>0$. There exist increasing sequences $(m_n^\E)_n$ and $(m_n^\F)_n$ so that $\sp{\E}(Kt,m_n^\E)\leq\sp{\E}(Kt)+\varepsilon$ and $\sp{\F}(Kt,m_n^\F)\leq \sp{\F}(Kt)+\varepsilon$. 

Note that $P_{m_n^\E}^\E\otimes P_{m_n^\F}^\F$ is a linear projection on $X\iten Y$. Moreover, the set of all $u\in S_{X\iten Y}$ such that  $(P_{m_n^\E}^\E\otimes P_{m_n^\F}^\F)u=u$ for some $n\in \mathbb N$ is dense in $S_{X\iten Y}$. Let $u\in S_{X\iten Y}$ be so that $(P_{m_n^\E}^\E\otimes P_{m_n^\F}^\F)u=u$ for some $n$. Consider the finite codimensional subspace $Z = \ker (P_{m_n^\E}^\E\otimes P_{m_n^\F}^\F)$. We claim that if $v\in Z$ and $||v||=t$ then
\[ \norm{u+v}\leq \norm{u+(I_X\otimes P_n^\F)v}(1+\sp{\F}(4K^2 t, m_n^\F) \]
Indeed, fix $x^*\in S_{X^*}$ and consider $y=(u+(I_X\otimes P_{m_n}^\F)v)(x^*)$. Note that $y \in \bigoplus_{i\leq {m_n^\F}} F_i$. We will distinguish two cases. Assume first that $\norm{y}\geq \frac{1}{2K}$. It follows that
\begin{align*}
\norm{(u+v)(x^*)} &= \norm{(u+(I_X\otimes P_{m_n^\F}^\F)(v))(x^*)+(I_X\otimes Q_{m_n^\F}^\F)(v)(x^*)}\\
				&= \norm{y}\cdot \norm{ \frac{y}{\norm{y}} + \frac{(I_X\otimes Q_{m_n^\F}^\F)(v)(x^*)}{\norm{y}}}  \\
  &\leq \norm{u+(I_X\otimes P_{m_n^\F}^\F)(v)}(1+\sp{\F}(4K^2 t,m_n^\F))\,,
\end{align*}
since $\norm{(I_X\otimes Q_{m_n^\F}^\F)(v)(x^*)}\leq 2Kt$. Now assume that $\norm{y}\leq \frac{1}{2K}$. Then 
\begin{align*}
\norm{(u+v)(x^*)} &= \norm{y+(I_X\otimes Q_{m_n^\F}^\F)(v)(x^*)}\\
				&\leq \frac{1}{2K} + 2K t \leq \frac{1}{K} \leq \norm{u+(I_X\otimes P_{m_n^\F}^\F)(v)}(1+\sp{\F}(4K^2 t,m_n^\F))
\end{align*}
since $t\leq \frac{1}{4K^2}$ and $\norm{u+(I_X\otimes P_{m_n^\F}^\F)v}\geq \frac{1}{K}$. The claim follows by taking supremum with $x^*\in S_{X^*}$. Now, take $y\in S_{Y^*}$ and consider $x=(u+(P_{m_n^\E}^\E\otimes P_{m_n^\F}^\F)v)(y^*)$. Apply the same argument and the fact that $(P_{m_n^\E}^\E\otimes P_{m_n^\F}^\F)v= 0$ to get that  
\begin{align*}
\norm{(u+(I_X\otimes P_{m_n^\F}^\F)v)(y^*)} &\leq \norm{u+(P_{m_n^\E}^\E\otimes P_{m_n^\F}^\F)v}(1+\sp{\E}(4K^2 t,m_n^\E))\\
&= (1+\sp{\E}(4K^2t, m_n^\E))\,,
\end{align*}
as desired. Thus,
\[ \norm{u+v}\leq (1+\sp{\E}(4K^2t)+\varepsilon)(1+\sp{\F}(4K^2t)+\varepsilon)\]
for every $\varepsilon>0$. Taking supremum with $v\in Z$, $||v||=t$ we get 
\[ \p_{X\iten Y}(t,u) \leq (1+\sp{\E}(4K^2t)+\varepsilon)(1+\sp{\F}(4K^2t)+\varepsilon)-1\,.\]
Finally, note that the above inequality holds for all $u$ in a dense subset of $S_{X\iten Y}$ and for every $\varepsilon>0$, so we are done. 
\end{proof}

From the above theorem we get a number of corollaries. 

\begin{corollary} \label{cor:AUSiten} Let $X$, $Y$ be strongly AUS spaces. Then $X\iten Y$ is AUS. If moreover $X$ and $Y$ are strongly AUS with power type $p$ and $q$, respectively, then $X\iten Y$ is AUS with power type $\min\{p,q\}$.
\end{corollary}

\begin{proof}[Proof of Theorem~\ref{th:usitenaus}] It follows from Proposition~\ref{prop:UCimplySAUC} and Corollary~\ref{cor:AUSiten}.
\end{proof}

\begin{remark} The injective tensor product of strongly AUC spaces need not to be AUC. Indeed, $\ell_2\iten \ell_2$ contains a subspace isometric to $c_0$, namely $\cvspan\{e_n\otimes e_n: n\in \mathbb N\}$ where $(e_n)_n$ denotes the standard basis of $\ell_2$, and so it is not AUC. On the other hand, it is proved in~\cite{DKLRZ13} that $\ell_p\iten \ell_q$ is AUC whenever $p,q<2$. We do not know if a similar statement holds for a more general class of Banach spaces. 
\end{remark}

Recall that a Banach space admits an equivalent AUS norm if and only if its Szlenk index, $\Sz(X)$, is less or equal than $\omega_0$ (see~\cite{KOS99,R13}). By a result of Schlumprecht~\cite{S15}, every Banach space with separable dual embeds into a Banach space with a shrinking basis and the same Szlenk index. Together with the separable determination of the Szlenk index, this provides another proof of the following particular case of Theorem 1.3 in~\cite{C17}: $X\iten Y$ admits an equivalent AUS norm whenever $X$ and $Y$ admit equivalent AUS norms.

\begin{corollary}\label{cor:kAUS} Let $X, Y$ be Banach spaces such that $X^*$ and $Y$ are strongly AUS. Then $\KK(X,Y)$ is AUS. If moreover $X^*$ is strongly AUS with power type $p$ and $Y$ is strongly AUS with power type $q$ then $\KK(X,Y)$ is AUS with power type $\min\{p,q\}$.
\end{corollary}

\begin{proof} Note that if $Y$ has an FDD then $X^*\iten Y$ is isometric to $\KK(X, Y)$ and apply Corollary~\ref{cor:AUSiten}.
\end{proof}

\begin{proof}[Proof of Theorem~\ref{th:uskaus}] It follows readily from Proposition~\ref{prop:UCimplySAUC}, Corollary~\ref{cor:kAUS} and the duality between uniform convexity and uniform smoothness.
\end{proof}

\begin{corollary}\label{cor:weak*AUC} Let $X, Y$ be strongly AUS spaces. Then $\NN(X,Y^*)$ is weak* AUC. If moreover $X$ and $Y$ are strongly AUS with power type $p$ and $q$, respectively, then $\NN(X,Y^*)$ is weak* AUC with power type $\max\{p',q'\}$.
\end{corollary}
\begin{proof} 
Note that $Y^*$ is separable since $Y$ admits a shrinking FDD by Proposition~\ref{prop:AUSshrinking}. By a result of Grothendieck, the spaces $(X\iten Y)^*$ and $\NN(X,Y^*)$ are isometric. Now the result follows from Corollary~\ref{cor:AUSiten} and the duality between AUS and weak* AUC norms.  
\end{proof}

By a result of Van Dulst and Sims~\cite{DS83}, the weak* AUC property for a dual space $X^*$ implies the weak* fixed point property, i.e., that every nonexpansive mapping from a weak*-closed bounded convex subset of $X^*$ into itself has a fixed point. 

\begin{corollary}\label{cor:weak*FPP} Let $X$, $Y$ be Banach spaces with strongly AUS norms. Then $\NN(X,Y^*)$ has the weak* fixed point property.
\end{corollary}

\section{Orlicz and Lorentz sequence spaces}\label{sec:orlicz}

We recall that an \emph{Orlicz function} $M$ is a continuous nondecreasing
 convex function defined on $\mathbb{ R}^+$ such that $M(0) = 0$ and $\lim_{t\to+\infty} M(t)=+\infty$. An Orlicz function is said to satisfy the $\Delta_2$-condition at zero if
\[ \limsup_{t\to 0} \frac{M(2t)}{M(t)} < +\infty \]
Every Orlicz function $M$ such that $\lim_{t\to+\infty} M(t)/t = +\infty$ has associated another Orlicz function $M^*$, which is its dual Young function, i.e.
\[ M^*(u) = \sup\{uv - M(v): 0 < v < +\infty\}\,.\]
To any Orlicz function $M$ we associate the space $h_M$ of all sequences of
scalars $(x_n)_n$ such that $\sum_{n=1}^\infty M(|x_n|/\rho)<+\infty$ for all $\rho>0$. 
The space $h_M$ endowed with the \emph{Luxemburg norm},
\[ \norm{x}_M = \inf\{ \rho>0 : \sum_{n=1}^\infty M(|x_n|/\rho)\leq 1\} \]
is a Banach space. A convexity argument (see Lemma 1.2.2 in~\cite{BM10}) yields that $\sum_{n=1}^\infty M(|x_n|)\leq \norm{x}$ if $\norm{x}\leq 1$, and $\sum_{n=1}^\infty M(|x_n|)\geq 1$ if $\norm{x}\geq 1$, for every $x\in h_M$. 

The \emph{Boyd indices} of an Orlicz function $M$ are defined as follows:
\begin{align*}
\alpha_M &= \sup\{q : \sup_{0<u,v\leq 1} \frac{M(uv)}{u^qM(v)}<+\infty\}, &
\beta_M &= \inf\{q : \inf_{0<u,v\leq 1} \frac{M(uv)}{u^qM(v)}>0\}
\end{align*}
It is easy to check that $1\leq \alpha_M\leq \beta_M\leq +\infty$, and $\beta_M<+\infty$ if and only if $M$ satisfy the $\Delta_2$ condition at zero. Moreover, the space $\ell_p$, or $c_0$ if $p=\infty$, is isomorphic to a subspace of an Orlicz space $h_M$ if and only if $p\in [\alpha_M, \beta_M]$. 

It was shown in~\cite{GJT07} that the space $h_M$ is AUS if $\alpha_M>1$. Moreover, $\alpha_M$ is the supremum of the numbers $\alpha>0$ such that the modulus of asymptotic smoothness of $h_M$ is of power type $\alpha$. In addition, Borel-Mathurin proved in~\cite{BM10} that if $\beta_M<+\infty$ then $h_M$ is AUC, and $\beta_M$ is the infimum of the numbers $\beta>0$ such that its modulus of asymptotic convexity is of power type $\beta$. A similar result was proved by Delpech in~\cite{Delpech09}. Moreover, their proofs actually show that $h_M$ is strongly AUC (resp. strongly AUS) whenever it is AUC (resp. AUS). 

\begin{proposition} \label{prop:SAUSorlicz} Let $M$ be an Orlicz function. If $\alpha_M>1$ then $h_M$ is strongly AUS with respect to the standard basis $\E=(e_n)_n$. Moreover, $\alpha_M$ is the supremum of the numbers $\alpha>0$ such that $\sp{\E}$ is of power type $\alpha$.
\end{proposition}
\begin{proof} Let $1<\alpha<\alpha_M$. Then there exists $C>0$ such that $M(uv)\leq C u^\alpha M(v)$ for every $0<u,v<1$. We will show that $\sp{\E}(t,n)\leq C t^\alpha$ for every $n$ and $0<t<1$. For that, let $x\in \vspan\{e_1,\ldots, e_n\}$ and $y\in\cvspan\{e_{i}: i>n\}$ with $\norm{x}=\norm{y}=1$. Note that $\norm{x+ty}\geq 1$ since $\E$ is monotone. Moreover, we may assume that $M(1)=1$ and thus $|y_n|\leq \norm{y}$ for each $n$. Therefore,
\begin{align*} \norm{x+ty}&\leq \sum_{i=1}^\infty M(|x_i+t y_i|) = \sum_{i=1}^n M(|x_i|) + \sum_{i=n+1}^\infty M(t|y_i|) \\
&\leq 1 + Ct^\alpha \sum_{i=n+1}^\infty M(|y_i|) = 1 + Ct^\alpha\,,  
\end{align*}  
as desired. 
\end{proof}

\begin{proposition} \label{prop:SAUCorlicz} Let $M$ be an Orlicz function. If $\beta_M<+\infty$ then $h_M$ is strongly AUC with respect to the standard basis $\E=(e_n)_n$. Moreover, $\beta_M$ is the infimum of the numbers $\beta>0$ such that $\sd{\E}$ is of power type $\beta$.
\end{proposition}

\begin{proof} Let $\beta>\beta_M$. Then there exists $C>0$ such that $M(uv)\geq C u^\beta M(v)$ for every $0<u,v<1$. Now use the monotony of $\E$ and mimic the proof of Lemma 1.3.10 in~\cite{BM10} to get that $\sd{\E}(t,n)\geq Ct^\beta$ for each $n\in \N$ and $0<t<1$. 
\end{proof}

\begin{proof}[Proof of Theorem~\ref{th:kAUSorlicz}]
First, note that $\KK(h_M, h_N)$ contain subspaces isometric to $h_M^*$ and $h_N$. Thus, if either $\alpha_M=1$, $\alpha_N=1$ or $\beta_M=+\infty$ then $\KK(h_M, h_N)$ contains a quotient isomorphic to $\ell_1$ or $\ell_\infty$ and therefore it is not even AUS renormable. Now assume that $\alpha_M>1, \beta_M<\infty$ and $\alpha_N>1$. First, by Proposition~\ref{prop:SAUCorlicz}, $h_M$ is strongly AUC with respect to the standard basis $\E=(e_n)_n$ with power type $\beta$ for each $\beta>\beta_M$. Moreover, since $M$ satisfy the $\Delta_2$ condition at $0$ we have that $\E$ is an unconditional basis of $h_M$. Note that $\ell_1$ is not isomorphic to a subspace of $h_M$ and thus, by a theorem of James, $\E$ is a shrinking basis of $h_M$ that it is also monotone. Thus we can apply Proposition~\ref{prop:SAUSduality} to get that $h_M^*$ is strongly AUS with power type $\beta$ for each $\beta<\beta_M'$. Finally, Proposition~\ref{prop:SAUSorlicz} implies that $h_N$ is strongly AUS with power type $\alpha$ for each $\alpha<\alpha_N$. Now it is enough to apply Corollary~\ref{cor:kAUS}. 
\end{proof}

Lennard proved in~\cite{L90} that the trace class operators $\NN(\ell_2,\ell_2)$ has the weak* fixed point property. This result was extended by Besbes~\cite{B92} to $\NN(\ell_p,\ell_q)$ with $p^{-1}+q^{-1}=1$. Moreover, it is shown in~\cite{DKLRZ13} that the same is true for $1<p,q<\infty$. 

\begin{corollary} Let $M,N$ be Orlicz functions such that $\alpha_M,\alpha_N>1$ and $ \beta_N<\infty$. Then the space $\NN(h_M,h_N)$ has the weak* fixed point property.
\end{corollary}

\begin{proof} Note that $h_N$ is reflexive since $1<\alpha_N,\beta_N<\infty$. Thus the canonical basis $(e_n)_n$ of $h_N$ is shrinking and monotone. Since $h_N$ is strongly AUC with respect to $(e_n)_n$, we get that $h_N^*$ is strongly AUS. Thus, we can apply Corollary~\ref{cor:weak*FPP}.
\end{proof}

Next result provides a characterisation of Orlicz functions $M,N$ such that the space $\KK(h_M, h_N)$ is NUS. This should be compared with~\cite[Corollary 4.4]{DKLRZ13}.

\begin{corollary}  Let $M,N$ be Orlicz functions. Then the following statements are equivalent:
\begin{itemize}
\item[i)] $1<\alpha_N$, $\beta_N<\alpha_M$ and $\beta_M<\infty$.
\item[ii)] $\KK(h_M,h_N)$ is reflexive. 
\item[ii)] $\KK(h_M,h_N)$ is NUS.
\end{itemize}
\end{corollary}

\begin{proof}
The equivalence between $i)$ and $ii)$ was shown in~\cite{AO97}. Since each NUS space is reflexive, $iii)$ implies $ii)$. Finally, if $i)$ and $ii)$ holds then $\KK(h_M,h_N)$ is AUS and reflexive and thus it is NUS.
\end{proof}

Finally, we will provide a result on strong asymptotic uniform convexity in Lorentz sequence spaces. Let us recall their definition. Let $1\leq p<\infty$ and let $w$ be a non-increasing sequence of positive numbers such that $w_1=1$, $\lim_n w_n = 0$ and $\sum_{n=1}^\infty w_n=\infty$. The \emph{Lorentz sequence space} $d(w,p)$ is defined as
\[ d(w,p) = \{x=(x_n)_n \in c_0 : \norm{x} = \sup_\sigma \left(\sum_{n=1}^\infty |x_{\sigma(n)}|^p w_n\right)^{1/p} <\infty \}\]
where $\sigma$ ranges over all permutations of the natural numbers. We refer the reader to~\cite{LT77} for more information about these spaces.

\begin{proposition} Let $d(w,p)$, $1<p<\infty$ be a Lorentz sequence space. Let $S_n = \sum_{i=1}^n w_i$. The following conditions are equivalent:
\begin{itemize} 
\item[i)] $d(w,p)$ is uniformly convex.
\item[ii)] $d(w,p)$ is strongly AUC. 
\item[iii)] $d(w,p)$ is AUC. 
\item[iv)] $\inf_n \frac{S_{2n}}{S_n}> 1$. 
\end{itemize}
\end{proposition}
\begin{proof}
The equivalence between $i)$ and $iv)$ was shown by Altshuler in~\cite{A75}. Note that the canonical basis $\E=(e_n)_n$ of $d(w,p)$ is monotone. Thus, $i)\Rightarrow ii)$ follows from Proposition~\ref{prop:UCimplySAUC}. Moreover, $d(w,p)$ is reflexive since $p>1$ and so $ii)\Rightarrow iii)$ follows from Corollary~\ref{cor:SAUSimplyAUS}. Finally, assume that $\inf_n \frac{S_{2n}}{S_n} = 1$ and let us show that $d(w,p)$ is not AUC. Fix $\varepsilon>0$ and take $n\in \N$ such that $\frac{S_{2n}}{S_n} < 1+\varepsilon$. Consider the sequence of unitary vectors given by $x_k=\sum_{i=k}^{k+n} \frac{1}{S_n^{1/p}}e_i$. Since $\E$ is a shrinking basis, we have that $(x_k)_k$ is weakly null. In addition,
\[ \norm{x_1+tx_k}^p = \sum_{i=1}^n \frac{1}{S_n} w_i + \sum_{i=n+1}^{2n} \frac{1}{S_n} w_i = \frac{S_{2n}}{S_n} \]
whenever $k>n$. Thus
\[ \overline{\delta}_{d(w,p)}(t,x_1) \leq \liminf_{k\to\infty} \norm{x_1+tx_k} -1\leq \varepsilon, \]
which finishes the proof. 
\end{proof}

\section{Strict convexity}

Finally we study strict convexity of $\KK(X, Y)$ and $X\iten Y$ by using John's ellipsoid theorem. Let us recall that a Banach space $X$ is \emph{strictly convex} if given $x, y \in S_X$, with $x \neq y$ then $\norm{x+y}<2$.

\begin{proof}[Proof of Proposition~\ref{th:KKstrictlyconvex}] First we will show that $\KK(X, Y)$ is not strictly convex. For that, let $Z$ be a $2$-dimensional subspace of $X^*$ and consider $Q$ the canonical projection from $X$ onto $X_0:= X/Z_\bot$. Note that $X_0$ is also $2$-dimensional since $X_0^*$ is isometric to $Z$. Let $Y_0$ be any $2$-dimensional subspace of $Y$ and denote $\iota:Y_0\to Y$ the inclusion operator. Since $Q(B_X)=B_{X_0}$, it follows that $T\mapsto \iota \circ T\circ Q$ defines an isometry from $\KK(X_0, Y_0)$ onto a subspace of $\KK(X, Y)$. Therefore, it suffices to show that $\KK(X_0,Y_0)$ is not strictly convex. For that we will identify the vectors of $X_0$ and $Y_0$ with points in $\mathbb{ R}^2$. Let $D_{X_0}$ be the ellipsoid of maximum volume containing $B_{X_0}$ and $D_{Y_0}$ be the ellipsoid of minimum volume contained in $B_{Y_0}$.  Moreover, fix $x_1\in B_{X_0}\cap \partial D_{X_0}$ and $y_1\in S_{Y_0}\cap D_{Y_0}$. Consider a linear map $T_{X_0}$ which transforms $D_{X_0}$ in $B_{\ell_2^2}$ and maps $x_1$ to $e_1$, and $T_{Y_0}$ which transforms $D_{Y_0}$ in $B_{\ell_2^2}$ and maps $y_1$ to $e_1$. For each $\alpha\in [-1,1]$, let $T_\alpha$ be the linear map given by $T_\alpha(ae_1+be_2)=ae_1+b\alpha e_2$. Finally, consider $R_\alpha = T_{Y_0}^{-1} \circ T_\alpha \circ T_{X_0}$. Then
\[ R_\alpha (B_{X_0}) \subset R_\alpha (D_{X_0}) \subset (T_{Y_0}^{-1} \circ T_\alpha) (B_{\ell_2^2}) \subset T_{Y_0}^{-1}(B_{\ell_2^2}) \subset D_{Y_0} \subset B_{Y_0} \]
so $\norm{R_\alpha}_{\KK(X_0,Y_0)} \leq 1$. Moreover, $R_\alpha (x_1) = (T_{Y_0}^{-1}\circ T_\alpha) (e_1) = T_{Y_0}^{-1}(e_1)= y_1$ and so $R_\alpha$ has norm one for each $\alpha \in [-1,1]$. Clearly $R_{1}\neq R_{-1}$, so this shows that $\KK(X_0,Y_0)$ is not strictly convex.

Finally, let $X_1$ be a $2$-dimensional subspace of $X$. The injective tensor product respects subspaces isometrically and thus $X_0\iten Y$ is isometric to a subspace of $X\iten Y$. Moreover, since $X_1$ is finite-dimensional we have that $X_1\iten Y$ is isometric to $\KK(X_1^*, Y)$ (see, e.g.~\cite[Corollary 4.13]{R02}), which is not strictly convex. This finishes the proof for $X\iten Y$.  
\end{proof}

\begin{remark} In~\cite{DK95} it is used Dvoretzky's theorem in order to show that none of the spaces $\LL(\ell_p,\ell_q)$ or $\LL(c_0,\ell_q)$ are superreflexive, that is, they do not admit an equivalent uniformly convex norm. Indeed, the same argument can be used to prove that neither $\KK(X,Y)$ or $X\iten Y$ are superreflexive whenever $X$ and $Y$ are infinite-dimensional. 
\end{remark}

\textbf{Acknowledgements}: The second-named author thanks Eva Gallardo Guti\'errez for awaring him about the work of Lennard ten years ago.

\bibliographystyle{siam}
\bibliography{aus}

\end{document}